\documentclass[a4paper,11pt]{amsart}
\usepackage{hyperref,latexsym}
\usepackage{enumerate}
\usepackage{graphicx}
\usepackage{float}
\theoremstyle{plain}
\newtheorem{theorem}{Theorem}[section]
\newtheorem{lemma}[theorem]{Lemma}

\theoremstyle{definition}

\theoremstyle{remark}

\begin{document}

\title[Outer billiards]
      {Outer billiards with the dynamics of a standard shift on a finite number of invariant curves}

\date{September 2018}
\author{Misha Bialy, Andrey E. Mironov, Lior Shalom }
\address{M. Bialy, School of Mathematical Sciences, Tel Aviv
University, Israel} \email{bialy@post.tau.ac.il}
\address{A.E. Mironov, 	Novosibirsk State University,
	Pirogova st 1, and Sobolev Institute of Mathematics,
	4 Acad. Koptyug ave., 630090 Novosibirsk, 
 Russia}
\email{mironov@math.nsc.ru}\address{L. Shalom, School of Mathematical Sciences, Tel Aviv
	University, Israel}\email{shalom@mail.tau.ac.il}
\thanks{M.B. and L.S were supported in part by ISF grant 162/15 and A.E.M. was supported by the Laboratory of Topology and Dynamics, Novosibirsk State University (contract no. 14.Y26.31.0025 with the Ministry of Education and Science of the Russian Federation).
	 It is our pleasure to thank
	these funds for the support}

\subjclass[2010]{} \keywords{{Outer billiards, Gutkin billiards, Total integrability, Chaotic behavior}}

\begin{abstract} We give a beautiful explicit example of a convex plane curve such that the outer billiard has a given finite number of invariant curves. Moreover, the dynamics on these curves is a standard shift. This example can be considered as an outer analog of the so-called Gutkin billiard tables. We test total integrability of these billiards,  in the region between the two invariant curves. Next, we provide computer simulations on the dynamics in this region. At first glance, the dynamics looks regular but by magnifying the picture we see components of chaotic behavior near the hyperbolic periodic orbits. We believe this is a useful geometric example for coexistence of regular and chaotic behavior of twist maps. 
\end{abstract}

\maketitle



\section{\bf Outer billiards} Let us remind first the model of outer billiard.
Consider a closed strictly convex oriented planar curve $\gamma$. The outer billiard transformation $T$ acts on $\Omega$ --- the vicinity of $\gamma$ by the rule:
A point $A$ is mapped to $ T(A)$ iff the segment $[A,T(A)]$ is tangent to $\gamma$ precisely at the middle of $[A,T(A)]$ and has positive orientation at the tangency point (Fig.\ref*{fig:outer}). 

\begin{figure}[h]
	\centering
	\includegraphics[width=9cm]{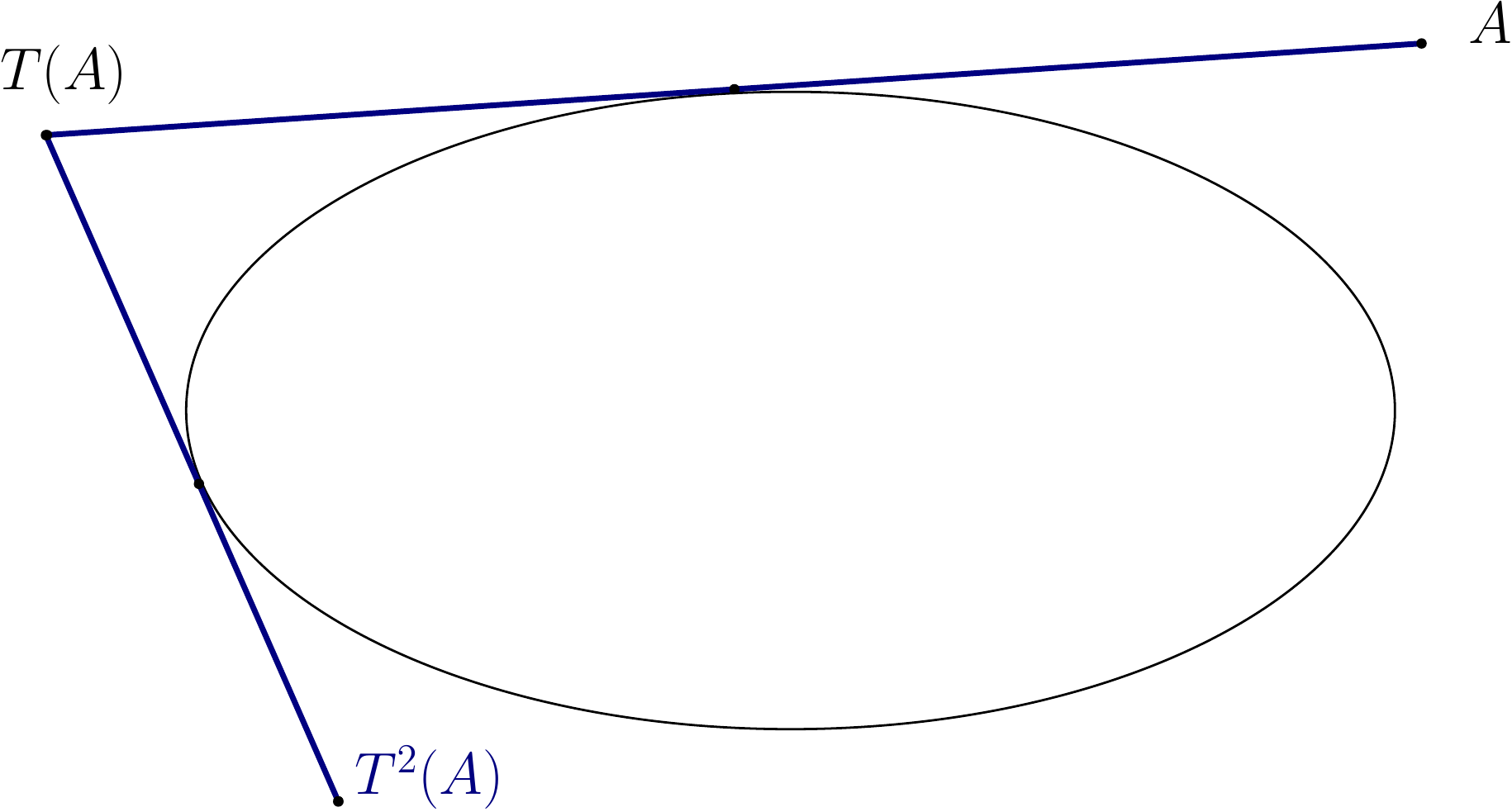}
	\caption{Outer billiard}
	\label{fig:outer}
\end{figure}

This map is a diffeomorphism of $\Omega$ preserving the standard area form (we refer to \cite{Boyland}, \cite{T2} for more details on outer billiards).
It is an open problem if there exist smooth convex closed curves other than ellipses which demonstrate integrable behavior.
The first result was obtained in \cite{Tab} and was completely settled in \cite {gl-sh}. It follows from these results that the answer is No if one restricts to polynomial integrals. It is not known however if the restriction of polynomiality  can be relaxed. Also, the following question is still open: 
Do there exist totally integrable outer billiards other than ellipses. Here, total integrability means the existence of a foliation of $\Omega$ consisting of invariant curves of $T$ running around $\gamma$.
In this paper we suggest a remarkable curve $\gamma$ such that the outer billiard demonstrates both regular and chaotic behaviors. For this curve we find some invariant curves and perform computer experiments. Recently an interesting "locally"
integrable behavior was numerically observed for a certain Birkhoff billiard in \cite{tresch}.


Let us choose now a regular parameter $t$  on $\gamma$. With this choice we have natural coordinates $(t,\lambda)$ in $\Omega$. Namely $(t,\lambda), \lambda>0$ corresponds to the point $$(t,\lambda), \lambda>0\mapsto
M=\gamma(t)+\lambda \dot\gamma (t).
$$
In these coordinates the area form takes the form
$$
\omega=[\dot\gamma(t),\ddot\gamma(t)]\lambda d\lambda\wedge dt,
$$
so the pair $(p,t), p=[\dot\gamma(t),\ddot\gamma(t)]\lambda^2/2$ form a symplectic pair.
Here and below the brackets $[\cdot,\cdot]$ denote the determinant of two vectors.
Moreover, the choice of parameter $t$ enables us to use a generating function 
which computes the area of the cup body: $$S(t,s)=Area({\rm \textit{Conv}}(M(t,s)\cup\gamma)),$$ see Fig. 1, where $Conv$ is the convex hull.
This is a generating function of $T$ because an easy computation of partial derivatives of $S$ gives the following formulas (here and below the subindex indicate partial derivatives):
\begin{equation}\label{s1s2}
S_1= - [\dot\gamma(t),\ddot\gamma(t)]\lambda^2/2,\quad S_2= [\dot\gamma(s),\ddot\gamma(s)]\mu^2/2.
\end{equation}
Notice that it follows from the definitions that $\frac{\partial\lambda}{\partial s}>0$, hence for the cross derivative we get from (\ref{s1s2}):
\begin{equation}\label{twist}
S_{12}<0,
\end{equation}
which is the so-called twist condition. Consider now a point $M\in\Omega$. We can write $M$ in two ways, see Fig. \ref{1}:
$$
M=\gamma(t)+\lambda\dot\gamma (t)=\gamma(s)-\mu \dot\gamma (s);\ \lambda,\mu>0.
$$
One can easily find:
\begin{equation}\label{mu}
\lambda=\frac{[\gamma(s)-\gamma(t),\dot\gamma(s)]}{[\dot\gamma(t),\dot\gamma(s)]},\quad \mu=\frac{[\gamma(t)-\gamma(s),\dot\gamma(t)]}{[\dot\gamma(t),\dot\gamma(s)]}.
\end{equation}
The following claim follows immediately from the definitions:
\begin{lemma} \label{lambda}
	For a positive constant $r$ let $\gamma_r$ be defined by 
	$$
	\gamma_r(t)=\gamma(t)+r\dot\gamma(t).
	$$
	Then $\gamma_r$ stays invariant under the outer billiard map $T$ if and only if for any point $M\in\gamma_r$ 
	\begin{equation}\label {constant}
	\mu(M)=r.
	\end{equation}
	\end{lemma}
\begin{figure}[h]
	\centering
	\includegraphics[width=0.5\linewidth]{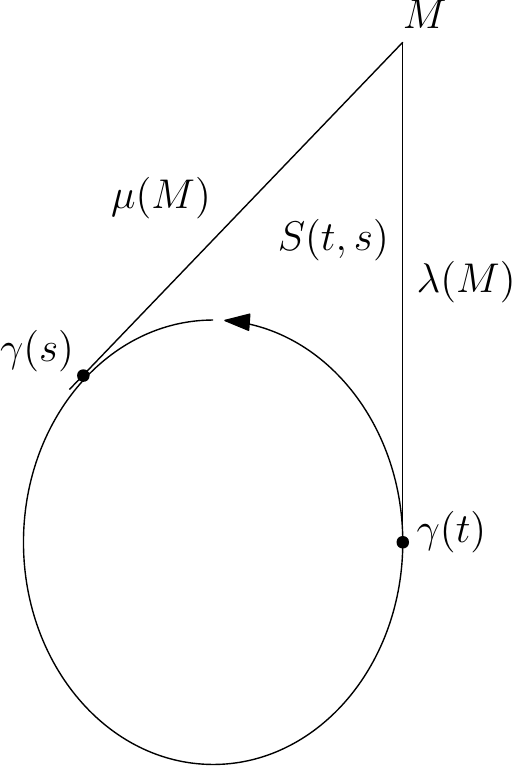}
	\caption{Outer billiard quantities}
	\label{1}
\end{figure}
\section{\bf Main Theorem}
Let us consider the following parametrized curve:
\begin{equation}\label{gam}
\gamma(t)=\left(\begin {array}{c}
x(t)\\
y(t)
\end{array}
\right)=\left(\begin {array}{c}
\cos t\\
\sin t
\end{array}
\right)+\epsilon \left(\begin {array}{c}
\cos nt\\
\sin nt
\end{array}
\right), \   t\in [0,2\pi],\ n\in \bf N.
\end{equation}
This is a perturbation of the  unit circle and hence is a closed convex curve at least for small positive $\epsilon$.

\begin{theorem}\label{main1}
Let $x \in (0,\pi/2)$ be any solution of the equation
\begin{equation}\label{gutkin}
\tan nx=n\tan x. 
\end{equation}
Set 
$$
r=\tan x
$$	
Then the curve $$
\gamma_r(t)=\gamma(t)+r\dot\gamma(t).
$$ is an invariant curve of billiard map $T$.
Moreover, the dynamics on this invariant curve is the standard shift:
$$
T(\gamma_r(t))=\gamma_r(s),\quad s=t+2x.
$$
\end{theorem}
{\it Remarks.}

	1. For exactly this $\gamma$ the ordinary Birkhoff billiard was studied for $n=2$ in  \cite {robnik}, \cite {strelc}.
	In this paper we consider outer billiard for $\gamma$ for higher values of $n$.
	
	2. In the paper \cite{T3} outer billiard tables having invariant curves of 3-periodic points were studied. In the present paper the rotation number on all the invariant curves given by Theorem \ref{main1} are always incommensurate with $\pi$ as it was proved in \cite {van}. 

3. The curve $\gamma$ gives a remarkable example of outer billiard having a number of invariant curves such that the dynamics on each of them is a standard shift. With this respect, this serves as an analog of the so-called Gutkin billiard tables of ordinary Birkhoff billiards \cite{gu1}. 
\medskip

{\it Proof of Theorem \ref{main1}.}
 Take any $t$ and $s=t+2x$. We claim that formulas (\ref{mu}) yield: 
	\begin{equation}\label{tricky}
	\lambda=\mu= \tan x
	\end{equation}
	identically for all $t$. This claim would finish the proof of Theorem \ref{main1} in view of Lemma \ref{lambda}.
	
	The verification of (\ref{tricky}) uses the explicit formulas (\ref{mu})  and tricky manipulations with trigonometric formulas using the equation (\ref{gutkin}) as we turn to explain. Let
	$$
	x(t)=\cos(t)+\varepsilon \cos(nt)  \quad  y(t)=\sin(t)+\varepsilon \sin(nt)
	$$	
	and
	\begin{equation}
	P_1=[\gamma(t)-\gamma(s),\dot\gamma(s)]=\det \left(\begin {array}{cc}
	x(t)-x(s)  & y(t)-y(s)\\
	x'(s) & y'(s)
	\end{array}
	\right),
	\end{equation}
	\begin{equation}
	P_2=[\dot\gamma(t),\dot\gamma(s)]=\det \left(\begin {array}{cc}
	x'(t) & y'(t)\\
	x'(s) & y'(s)
	\end{array}
	\right).
	\end{equation}
	Equation (\ref{tricky}) follows from the following lemma.
	\begin{lemma} 
		If $s=t+2x$ and $x$ satisfies the equation  (\ref{gutkin}),
		then the identity holds 
		$$
		P_1+P_2\tan(x)=0.
		$$
	\end{lemma}
	
	\begin{proof}
		From (\ref{gutkin}) we have
		\[
		\sin(2nx)=\frac{2n\tan(x)}{1+n^2\tan^2(x)} \quad \cos(2nx)=\frac{1-n^2\tan^2(x)}{1+n^2\tan^2(x)}
		\]
		This allows us to express $x(s), y(s), x'(s), y'(s)$ by:
		\\
		$$
		x(s)=\cos(t+2x)+\varepsilon\cos(n(t+2x))=
		$$
		$$
		\cos(t+2x)+\varepsilon(\cos(nt)\cos(2nx)-\sin(nt)\sin(2nx))=
		$$
		$$
		\cos(t+2x)+\varepsilon\left(\cos(nt)\frac{1-n^2\tan^2(x)}{1+n^2\tan^2(x)}-\sin(nt)\frac{2n\tan(x)}{1+n^2\tan^2(x)}\right)
		$$
		\\
		$$
		y(s)=\sin(t+2x)+\varepsilon\sin(n(t+2x))=
		$$
		$$
		\sin(t+2x)+\varepsilon(\sin(nt)\cos(2nx)+\cos(nt)\sin(2nx))=
		$$
		$$
		\sin(t+2x)+\varepsilon\left(\sin(nt)\frac{1-n^2\tan^2(x)}{1+n^2\tan^2(x)}+\cos(nt)\frac{2n\tan(x)}{1+n^2\tan^2(x)}\right)
		$$
		\\
		$$
		x'(s)=-\sin(t+2x)-n\varepsilon\sin(n(t+2x))=
		$$
		$$
		-\sin(t+2x)-n\varepsilon(\sin(nt)\cos(2nx)+\cos(nt)\sin(2nx))=
		$$
		$$
		- \sin(t+2x)-n\varepsilon\left(\sin(nt)\frac{1-n^2\tan^2(x)}{1+n^2\tan^2(x)}+\cos(nt)\frac{2n\tan(x)}{1+n^2\tan^2(x)}\right)
		$$
		\\
		$$
		y'(s)=\cos(t+2x)+n\varepsilon\cos(n(t+2 x))=
		$$
		$$
		\cos(t+2x)+n\varepsilon(\cos(nt)\cos(2nx)-\sin(nt)\sin(2nx))=
		$$
		$$
		\cos(t+2x)+n\varepsilon\left(\cos(nt)\frac{1-n^2\tan^2(x)}{1+n^2\tan^2(x)}-\sin(nt)\frac{2n\tan(x)}{1+n^2\tan^2(x)}\right)
		$$
		\\
		\\
		Using the above identities, we get:
		\\
		$$
		P_1+P_2\tan(x)=
		$$
		$$
		y'(s)(x(t)-x(s)+x'(t)\tan(x))-x'(t)(y(t)-y(s)+y'(t)\tan(x))
		$$
		\\
		The coefficient at $\varepsilon^0$ in $P_1+P_2\tan(x)$ is:
		\\
		$$
		\cos(t+2x)(\cos(t)-\cos(t+2x)-\sin(t)\tan(x))+\sin(t+2x)(\sin(t)-\sin(t+2x)+\cos(t)\tan(x))=
		$$
		$$
		\cos(t+2x)\cos(t)+\sin(t+2x)\sin(t)-1-(\cos(t+2x)\sin(t)-\sin(t+2x)\cos(t))\tan(x)=
		$$
		$$
		\cos(2x)-1+\sin(2x)\tan(x)=0
		$$
		\\
		The coefficient at $\varepsilon^1$ in $P_1+P_2\tan(x)$ is:
		\\
		$$
		\cos(t+2x)\Big(\cos(nt)-n\sin(nt)\tan(x)+\frac{-\cos(nt)+2n\sin(nt)\tan(x)+n^2\cos(nt)\tan^2(x)}{1+n^2\tan^2(x)}\Big)+
		$$
		$$
		n\Big(\frac{-2n\sin(nt)\tan(x)+\cos(nt)(1-n^2\tan^2(x))}{1+n^2\tan^2(x)}\Big)\Big(\cos(t)-\cos(t+2x)-\sin(t)\tan(x)\Big)+
		$$
		$$
		\sin(t+2x)\Big(\sin(nt)+n\cos(nt)\tan(x )+ \frac{-\sin(nt)-2n\cos(nt)\tan(x)+n^2\sin(nt)\tan^2(x)}{1+n^2\tan^2(x)}\Big)+
		$$
		$$
		n\Big(\frac{2n\cos(nt)\tan(x)+\sin(nt)(1-n^2\tan^2(x))}{1+n^2\tan^2(x)}\Big)	\Big(\sin(t)-\sin(t+2x)+\cos(t)\tan(x)\Big) 
		$$
		\\
		\\
		We will observe the terms as coefficients to $n$ in some power. We have:
		$$
		\Big(\frac{1}{1+n^2\tan^2(x)}\Big)
		$$
		$$
		\Big(n^3\big[-\cos(t+2x)\sin(nt)\tan^3(x)-\cos(nt)\tan^2(x)(\cos(t)-\cos(t+2x)-\sin(t)\tan(x))+
		$$
		$$
		\ \ \ \ \ \ \ \sin(t+2x)\cos(nt)\tan^3(x)-\sin(nt)\tan^2(x)(\sin(t)-\sin(t+2x)+\cos(t)\tan(x))\big]+
		$$
		$$
		n^2\big[2\cos(t+2x)\cos(nt)\tan^2(x)+2\sin(t+2x)\tan^2(x)\sin(nt)-2\sin(nt)\tan(x)(\cos(t)-
		$$
		$$
		\ \ \ \ \ \ \  \cos(t+2x)-\sin(t)\tan(x))+2\cos(nt)\tan(x)(\sin(t)-\sin(t+2x)+\cos(t)\tan(x))\big]+
		$$
		$$
		n\big[\cos(nt)\cos t+\sin(nt)\sin t-\cos (t+2x)\cos(nt)-\sin(t+2x)\sin(nt)+
		$$
		$$
		\ \ \ \ \ \ \ \ \ \tan(x)(\sin(nt)\cos t-\cos(nt)\sin t+\cos (t+2x)\sin(nt)-\sin(t+2x)\cos(nt))\big] \Big)
		$$
		\\
		All of the coefficients at $n^3, n^2, n^1$ in the previous expression are 0. \\
		Indeed, the coefficient at $n^3$ reads:
		\\
		$$
		\tan^2(x)(\sin(t+2x)\sin(nt)+\cos(t+2x)\cos(nt)-\sin(nt)\sin t-\cos(nt)\cos t )+
		$$
		$$
		\tan^3(x)(\sin(t+2x)\cos(nt)-\cos(t+2x)\sin(nt)+\cos(nt)\sin t-\sin(nt)\cos t)=
		$$
		$$
		\tan^2(x)(\cos(2x+t-nt)-\cos(t-nt))+\tan^3(x)(\sin(t-nt+2x)+\sin(t-nt))=
		$$
		$$
		\tan^3(x)(2\cos(x)\sin(t-nt+x))-\tan^2(x)(2\sin(x)\sin(t-nt+2x))=0
		$$
		\\
		Coefficient at $n^2$ reads:
		\\
		$$
		2\tan(x)(\cos(nt)\sin t-\sin(nt)\cos t+\cos (t+2x)\sin(nt)-\sin(t+2x)\cos(nt))+
		$$
		$$
		2\tan^2(x)(\cos(nt)\cos t+\sin(nt)\sin t+\cos(t+2x)\cos(nt)+\sin(t+2x)\sin(nt))=
		$$
		$$
		2\tan(x)(\sin(t-nt)-\sin(t-nt+2x))+2\tan^2(x)(\cos(t-nt)+\cos(t-nt+2x)))=
		$$
		$$
		-4\tan(x)\sin x\cos(t-nt+x)+4\tan^2(x)\cos x\cos(t-nt+x)=0
		$$
		\\
		Coefficient at $n$ reads:
		\\
		$$
		(\cos(nt)\cos t+\sin(nt)\sin t-\cos (t+2x)\cos(nt)-\sin(t+2x)\sin(nt))+
		$$
		$$
		\tan(x)(\sin(nt)\cos t-\cos(nt)\sin t+\cos (t+2x)\sin(nt)-\sin(t+2x)\cos(nt))=
		$$
		$$
		(\cos(t-nt)-\cos(t-nt+2x)-\tan(x)(\sin(t-nt)+\sin(t-nt+2x))=
		$$
		$$
		2\sin x\sin(t-nt+x)-2\tan x\cos x\sin(t-nt+x)=0
		$$
		\\
		The coefficient at $\varepsilon^2$ in  $P_1+P_2\tan(x)$ is:
		\\
		$$
		n\Big(\frac{-2n\sin(nt)\tan(x)+\cos(nt)(1-n^2\tan^2(x))}{1+n^2\tan^2(x)}\Big) 
		$$
		$$
		\Big(\cos(nt)-n\sin(nt)\tan(x)-\frac{\cos(nt)-2n\sin(nt)\tan(x)}{1+n^2\tan^2(x)}+\frac{n^2\cos(nt)\tan^2(x)}{1+n^2\tan^2(x)}\Big)+
		$$
		$$
		n\Big(\frac{2n\cos(nt)\tan(x)+\sin(nt)(1-n^2\tan^2(x))}{1+n^2\tan^2(x)}\Big)
		$$
		$$
		\Big(\sin(nt)+n\cos(nt)\tan(x)-\frac{\sin(nt)+2n\cos(nt)\tan(x)}{1+n^2\tan^2(x)}+
		\frac{n^2\sin(nt)\tan^2(x)}{1+n^2\tan^2(x)}\Big)=
		$$
		$$
		n\Big(\frac{-2n\sin(nt)\tan(x)+\cos(nt)(1-n^2\tan^2(x))}{1+n^2\tan^2(x)}\Big)
		$$
		$$
		\Big(\frac{n\sin(nt)\tan(x)+2n^2\cos(nt)\tan^2(x)-n^3\sin(nt)\tan^3(x)}{1+n^2\tan^2(x)}\Big)+
		$$
		$$
		n\Big(\frac{2n\cos(nt)\tan(x)+\sin(nt)(1-n^2\tan^2(x))}{1+n^2\tan^2(x)}\Big)
		$$
		$$
		\Big(\frac{-n\cos(nt)\tan(x)+2n^2\sin(nt)\tan^2(x)+n^3\cos(nt)\tan^3(x)}{1+n^2\tan^2(x)}\Big)=
		$$
		$$
		n\Big(\frac{-2n\sin(nt)\tan(x)+\cos(nt)(1-n^2\tan^2(x))}{1+n^2\tan^2(x)}\Big)
		$$
		$$
		\Big(n\tan(x)\frac{\sin(nt)(1-n^2\tan^2(x))+2n\cos(nt)\tan(x)}{1+n^2\tan^2(x)}\Big)+
		$$
		$$
		n\Big(\frac{2n\cos(nt)\tan(x)+\sin(nt)(1-n^2\tan^2(x))}{1+n^2\tan^2(x)}\Big)
		$$
		$$
		\ \ \ \ \Big(n\tan(x)\frac{2n\sin(nt)\tan(x)-\cos(nt)(1-n^2\tan^2(x))}{1+n^2\tan^2(x)}\Big)=0
		$$
		Lemma is proved.
	\end{proof}
This completes the proof of Theorem \ref{main1}.

\section{\bf Total integrability test}We consider the annulus between two neighboring invariant curves corresponding to two solutions $x_1,x_2$ of the equation (\ref{gutkin}).

\begin{figure}[h]
	\centering
	\includegraphics[width=10cm]{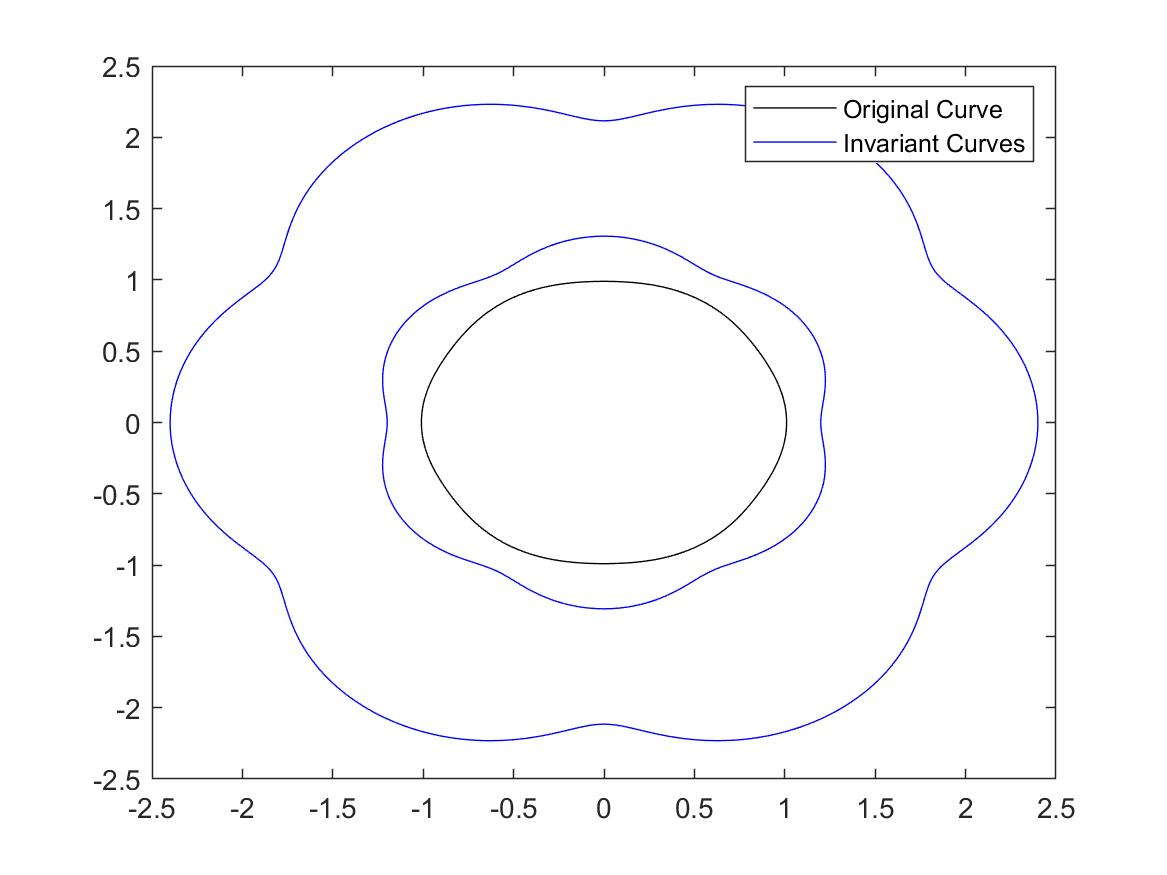}
	\caption{Two invariant curves bounding the domain $\Omega_{x_1x_2}$}
	\label{fig:TwoInvariantCurves}
\end{figure}

 We denote by $\Omega_{x_1x_2}$ the domain between these two invariant curves. We want to apply total integrability test as it was developed in \cite{B}, \cite {B0}.
Theory tells that total integrability of the map $T$ in the domain $\Omega_{x_1x_2}$ would imply the inequality
\begin{equation}\label{test}
\int_{\Omega_{x_1x_2}}(S_{11}+2S_{12}+S_{22})\ d(Area)\geq 0,
\end{equation}
where the Area form can be computed in terms of $S$ as
$$
d(Area)= -S_{12}\ dtds.
$$
Here we put "$-$" sign due to the twist condition (\ref*{twist}).
The domain of integration in coordinates $(t,s)$ looks by Theorem \ref{main1} as follows:
$$
\Pi=\{t\in[0, 2\pi],\  (s-t)\in [2x_1, 2x_2]\}.
$$
Thus altogether we get that the total integrability would imply:
\begin{equation}\label{test1}
I:=\int_{\Pi}(S_{11}+2S_{12}+S_{22})\ S_{12}\ dtds\leq 0.
\end{equation}

On the other hand we have the following experimental:
\begin{theorem}[Experimental]
	For  $n=7, \epsilon=0.01$ the solutions of equation \ref{gutkin} in the interval $(0,\pi/2)$, where are  $x_1 = 0.646471$ and $x_2 = 1.111932$ and the integral
	$$
	I=9.16345
	$$
\end{theorem}
{\it "Proof":}
This integral is explicit since the domain $\Pi$ is explicit and the derivatives of $S$ can be explicitly computed using the formulas (\ref{gam}),(\ref*{mu}). We couldn't however compute this integral exactly. MATHEMATICA software was used to evaluate it numerically.

\section{\bf Some periodic points}The curve $\gamma$ admits large number of symmetries. 
Therefore some periodic points for the outer billiard $T$ can be found explicitly as follows:
Take the following points on $\gamma$  
 $$P_k=\gamma( t_k);\quad t_k=\frac{2\pi k}{n-1};\quad k=0,1...,n-1,$$
 
  $$Q_k=\gamma( s_k);\quad s_k=\frac{\pi}{n-1}+\frac{2\pi k}{n-1};\quad k=0,1...,n-1.$$
 These points are staying on the maximal and minimal distance from the origin respectively.
 One can see that the points $E_k, H_k$ which are the intersection points of the  tangent lines to $\gamma$ at $P_k$ and $P_{k+2}$ and $Q_k, Q_{k+2}$ respectively are periodic points of $T$.
 On the Fig. \ref{fig:Full} one clearly see the points $E_k$ corresponding to six large elliptic islands and $H_k$ six hyperbolic periodic points between them. 
 
\begin{figure}[H]
	\centering
	\includegraphics[width=17cm]{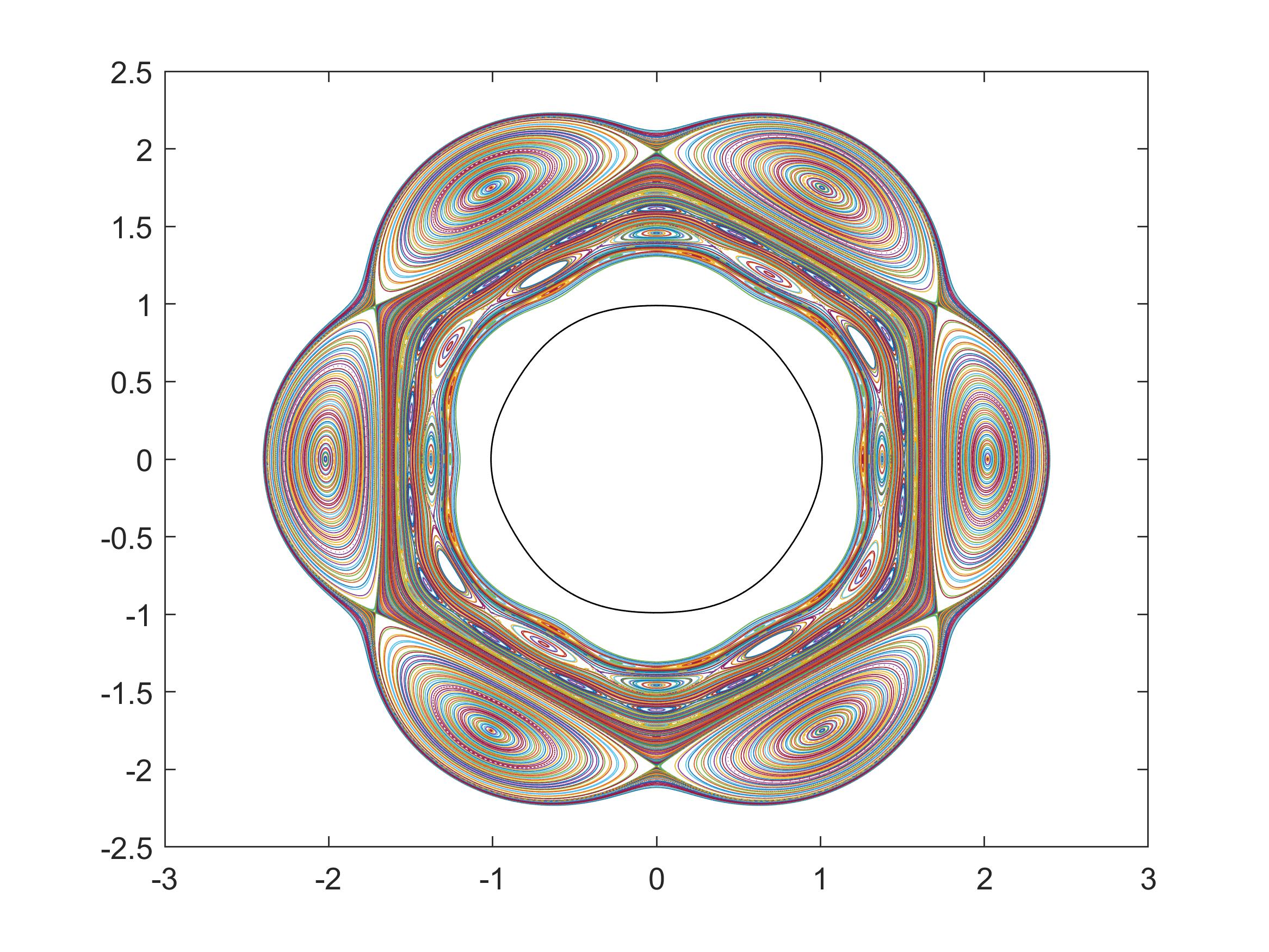}
	\caption{10,000 iterations of 500 initial points}
	\label{fig:Full}
\end{figure}

\section{\bf Pictures of Computer simulations}
In this section we show the results of computer simulations for the phase portrait of the mapping $T$ for $n=7, \epsilon=0.01$. They demonstrate clearly that there are contractible invariant curves (the islands) around elliptic periodic points. At the first stage of the computations there was an impression that the mapping $T$ could be integrable.

Later, on magnifying the picture between KAM curves and especially near the hyperbolic periodic points $H_k$, we found chaotic regions.
We believe that the map $T $ gives a very important geometric example 
of coexistence of regular and chaotic behaviors (we refer to an important survey papers  \cite{Pesin}  \cite{Strelcyn} on the coexistence problem).

\bigskip
\begin{figure}[H]
	\centering
	\includegraphics[width=16cm]{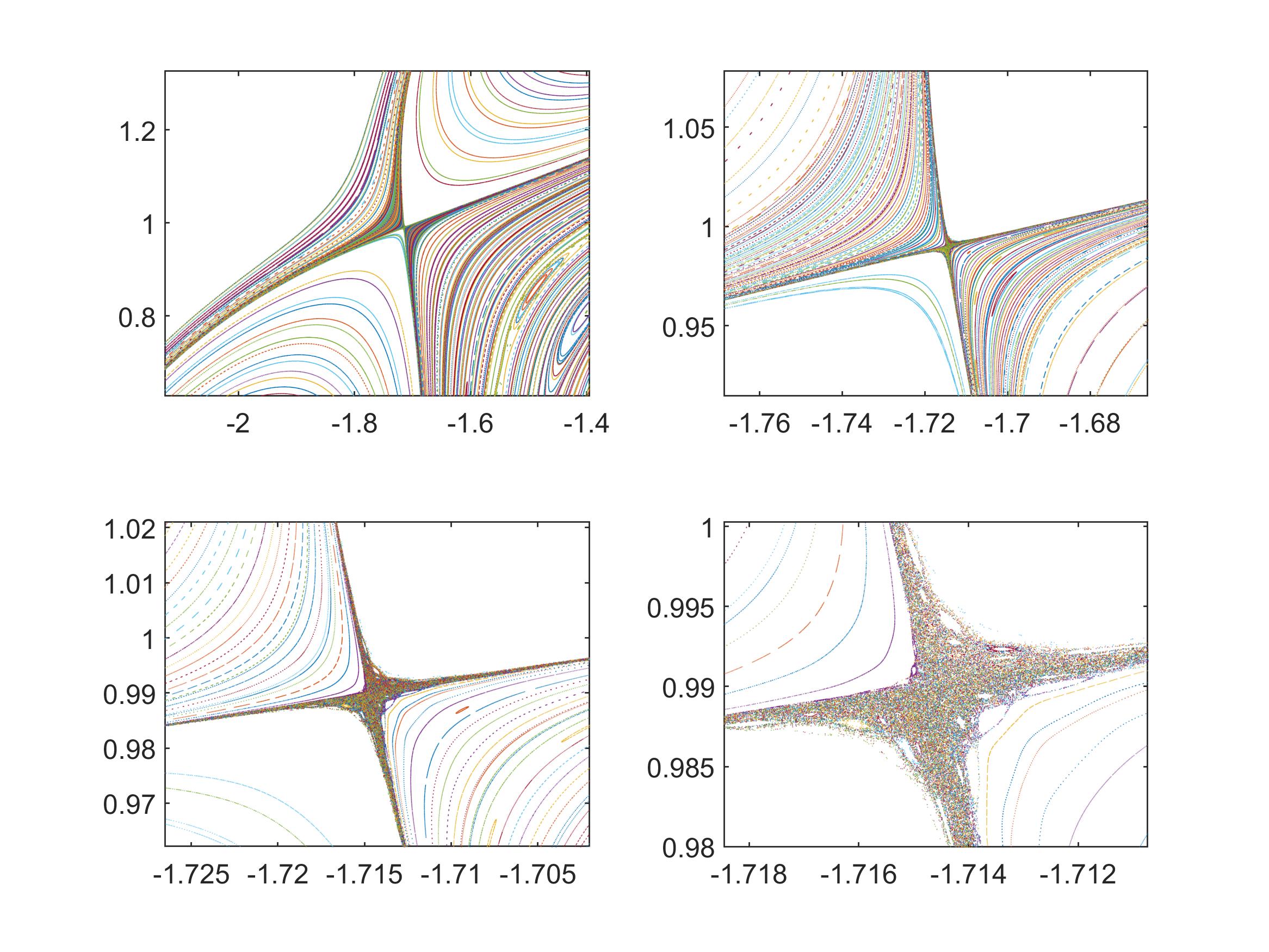}
	\caption{Magnifying a neighborhood of hyperbolic periodic point}
	\label{fig:Mag1}
\end{figure}
\section*{\bf Acknowledgments}This work was started during the XXXVII Workshop on Geometric Methods in Physics, BIAŁOWIEŻA, POLAND. We would like to thank the organizers for this opportunity.

\end{document}